\documentclass[15 pt]{amsart}
\usepackage{amsfonts, amsmath, amssymb, amscd, amsbsy, amscd, amsgen, amsopn, amstext, amsxtra}
\usepackage[latin5]{inputenc}
\usepackage{xcolor}
\usepackage{upgreek}

\newtheorem{thm}{Theorem}[section]
\newtheorem{cor}[thm]{Corollary}

\newtheorem{lem}[thm]{Lemma}

\newtheorem{pro}[thm]{Proposition}
\theoremstyle{definition}

\theoremstyle{question}
\newtheorem{que}[thm]{Question}
\theoremstyle{remark}

\theoremstyle{example}
\newtheorem{example}[thm]{Example}

\theoremstyle{note}

\newcommand{\FF}{\mathbb{F}}

\begin{document}

\title{A characterization of  abelian group codes in terms of their parameters }

\author{Fatma Altunbulak  Aksu and \.Ipek Tuvay}

\address{Department of Mathematics, Mimar Sinan Fine Arts University, \c{S}i\c{s}li, Istanbul, Turkey}

\email{fatma.altunbulak@msgsu.edu.tr, ipek.tuvay@msgsu.edu.tr}


\thanks{}


\date{\today}

\dedicatory{}

\commby{}

\maketitle
\begin{abstract}In 1979, Miller proved that for a group $G$ of odd order, two minimal group codes in $\mathbb{F}_2G$ are $G$-equivalent if and only they have identical weight distribution. In 2014, Ferraz-Guerreiro-Polcino Milies disprove Miller's result by giving an example of two non-$G$-equivalent minimal codes with identical weight distribution. In this paper, we give a characterization of finite abelian groups so that over a specific set of group codes, equality of important parameters of two codes implies the $G$-equivalence of these two codes. As a corollary, we prove that  two minimal codes with the same weight distribution 
are $G$-equivalent if and only if  for each prime divisor $p$ of $|G|$, the Sylow $p$-subgroup of 
$G$ is homocyclic.

\end{abstract}

\section{Introduction}
\label{Sec:1}

An abelian code over a field is an ideal in a finite group algebra of an abelian group. This definition is given by  Berman \cite{Berman} and MacWilliams \cite{MacWilliams1}.
More generally,   {\bf{a group code}} is  an ideal in a finite group algebra. One can easily show that cyclic codes are ideals in finite group algebras of cyclic groups. Reed-Muller codes are group codes for  elementary abelian $p$-groups
(see \cite{Berman2}). There are many  important linear codes which can be  viewed as  group codes \cite{KS}. It is proved in \cite{BW} that  group codes are   asymptotically good over any field. Besides this, group codes have more algebraic structures than linear codes. Because of all these, they are of interests for many researchers. 

There are parameters which measure the quality of codes: length, dimension and weight of codes. Let $G$ be an abelian group and $\FF$  a finite field of characteristic coprime to the order of $G$. For an element $\alpha \in \FF G$, the {\bf{support of $\alpha$}},
is the set $\mathrm{supp}(\alpha)=\{g \in G |~ \alpha_g\neq 0\}$ and the {{\bf{weight of $\alpha$}}} is $w(\alpha)=|~\mathrm{supp}(\alpha)~|$. If $I$ is an ideal in $\FF G$, then the {\bf{weight}} of $I$ is $w(I)= {\rm{min}} ~{\{w(\alpha)|~ \alpha \in I, \alpha\neq 0\}}$. Under some equivalence relations, these three parameters are preserved. One of these equivalences is $G$-equivalence.
 
 As defined in \cite{Miller},  two abelian codes 
$I$ and $J$ in $\mathbb{F}G$ are  {{\bf $G$-equivalent}} if there is a group automorphism $\varphi:G\rightarrow G$ whose linear extension to the group algebra maps $I$ onto $J$. If two codes are $G$-equivalent, then they have the same weight, the identical weight distribution and the same dimension. 
  
In \cite{Miller} (see Theorem 3.9), it is proved that for an abelian group of odd order, two minimal abelian codes in $\mathbb{F}_2G$ are $G$-equivalent if and only if they have the identical weight distribution. 
In \cite{FGM} (see Proposition IV.2), the authors  show that for $\mathbb{F}_2(C_9\times C_3)$, there are non-$G$-equivalent minimal codes having the identical weight distribution.

From Table IV in \cite{FGM}, we can conclude that, for cyclic groups,  equality of weight distrubition of codes gives the $G$-equivalence of codes. Now it is natural to ask for which abelian groups and fields, equality of weight distribution implies the $G$-equivalence of codes.

In this paper, by this motivation, for a semisimple finite abelian group algebra $\FF_q G$, we are concerned with the following conditions and characterize finite abelian groups satisfying these conditions.

Let $\mathcal{I}$ be an arbitrary set of codes in $\FF_q G$.

{\bf Condition A:}  Let   $I_1$, $I_2$ be in $\mathcal{I}$. $I_1$ is $G$-equivalent to $I_2$ if and only if $I_1$ and
$I_2$ have the same weight.

{\bf Condition B:}  Let $I_1$ and $I_2$ be in $\mathcal{I}$. $I_1$ is $G$-equivalent to $I_2$ if and only if $I_1$ and
$I_2$ have the identical weight distribution. 

Besides the  weight and the weight distribution of codes, we also consider dimension of codes.

{\bf Condition C:}  Let $I_1$ and $I_2$ be in $\mathcal{I}$. $I_1$ is $G$-equivalent to $I_2$ if and only if $I_1$ and
$I_2$ have the same dimension. 

Note that the forward direction of the conditions above always holds, because $G$-equivalence preserves the important
parameters such as dimension, weight and weight distrubition of the codes.

In this paper, for a specific family of ideals in $\FF_q G$, we characterize finite abelian groups for which Conditions A, B, C holds.
The structure of the paper is as follows. In section two, we give some needed material for the subject. In section three, we concentrate on the problem for $p$-groups. In the last part, we consider the problem for the composite groups.

\section{Preliminaries}
\label{Sec:2}

Let $G$ be a finite abelian group and $\FF_q$ a field such that $(q, |G|)=1$.
For an arbitrary subgroup $H\leq G$, the element 
 $$\widehat{H}= \frac{1}{|H|} \sum_{h\in H}h \in \FF_q G$$ is an idempotent element.

We say that $H\leq G$ is
{{\bf cocyclic}} if $H\neq G$ and $G/H$ is a cyclic group. When $G$ is an abelian $p$-group and $H$
is a cocyclic subgroup of $G$, there is a unique subgroup $H^*\leq G$ containing $H$ with $H^*/H\cong C_p$.

Suppose that $G$ is an abelian $p$-group. For a cocyclic subgroup $H<G$, there is a corresponding
idempotent $e_H=\widehat{H}-\widehat{H^*}$. When $H=G$, set $e_G=\widehat{G}$. 
When $G$ is an abelian composite  group, set $G=G_{p_1}\times \ldots \times G_{p_k}$
where $G_{p_i}$ is a Sylow $p_i$-subgroup of $G$ for any $i$. Any cocyclic $H$ of $G$ can be
written as
$$H=H_{p_1}\times \ldots \times H_{p_k}$$
where for any $i$, either $H_{p_i}$ is a cocyclic subgroup of $G_{p_i}$ or $H_{p_i}=G_{p_i}$. Then for such $H$, $$e_H=e_{H_{p_1}}\ldots e_{H_{p_k}}$$
is the corresponding idempotent. For $H=G$, set $e_G=\widehat{G}$. These idempotents are defined in \cite{FGM}. Let $\mathcal{S}_{cc}(G)$ denote the set of all cocyclic subgroups of $G$. By Proposition II.6 and Lemma II.7 in \cite{FGM}, the set 
$\{e_H| \ H\in \mathcal{S}_{cc}(G)\}\cup \{\widehat{G}\} $ is a set of orthogonal idempotents in $\FF_qG$ and we have 
$\sum\limits_{H \in \mathcal{S}_{cc}(G)  }e_H + \widehat{G}=1$.

For any finite abelian group $G$ and for a cocyclic subgroup $H< G$, set $I_H=(\FF_q G) e_H$,  $I_G=(\FF_qG)\widehat{G}$ and
$$\mathcal{I}_{\FF_qG}=\{ I_H \ | \ H < G \text{ is cocyclic or } H=G\}.$$

Let $p$ be a prime integer, $G$ a finite abelian group of exponent $p^n$ and $\FF_q$ a finite field so that $\FF_q G$ is semisimple. If $q$ is a generator of the unit group $U(\mathbb{Z}_{p^n})$, then $\mathcal{I}_{\FF_qG}$ is the set of minimal ideals in $\FF_q G$ (see Theorem 4.1 in \cite{FM}).

For any finite abelian group $G$, it is easy to see that $\{g\widehat{H}\ | \ gH\in G/H\}$ is a basis for $ (\FF_q G) \widehat{H}$
over $\FF_q$.
Recall that from Proposition 3.6.7 of \cite{MS} we have a ring isomorphism
$$\varphi_H: (\FF_q G) \widehat{H} \to \FF_q(G/H)$$
where $\varphi_H$ is defined to be the linear extension of the map $g \widehat{H} \mapsto gH$.

We begin with some  simple observations.

\begin{lem}\label{reduction1}
Let $G$ be a finite group and let $T$ be a normal subgroup of $G$. If $\alpha \in (\FF_q G) \widehat{T}$
then $w(\alpha)=|T| \ w(\varphi_T(\alpha)) $. Moreover, if $I\subseteq (\FF_q G) \widehat{T}$ then $w(I)=|T| \ w(\varphi_T(I)) $.
\end{lem}

\begin{proof}
Since $\mathcal{B}=\{g \widehat{T}\ | \ gT \in G/T\}$ is a basis for $ (\FF_q G) \widehat{T}$, $\alpha$ can be
written as follows
$$\alpha=\sum_{g\widehat{T} \in \mathcal{B}}\alpha_{gT} ~g \widehat{T}.$$
Also,  different elements in $\mathcal{B}$ have disjoint supports, so we have that
$$w(\alpha)= k |T|$$
where $k$ is the number of non-zero $\alpha_{gT}$ in the above presentation of $\alpha$. On the other hand,
$$\varphi_T(\alpha)=\sum_{g\widehat{T} \in \mathcal{B}}\alpha_{gT}~gT\in \FF_q(G/T)$$
has weight equal to $k$. Hence, the claim follows.
\end{proof}
Recall that two subgroups $H$ and $K$ are called {\bf $G$-isomorphic} if there is an automorphism $\theta$ of $G$ 
such that $\theta(H)=K$. 

\begin{lem}\label{equivGiso}
Let $G$ be a finite abelian group. Then two codes $I_{H}$ and $I_{K}$ in ${\mathcal{I}}_{{\FF}_q G}$ are $G$-equivalent 
if and only if $H$ and $K$ are $G$-isomorphic.
\end{lem}

\begin{proof} 
If $H$ and $K$ are $G$-isomorphic, then by definition there exists $\theta \in {\rm{Aut}}(G)$ such that $\theta(H)=K$. We will also use $\theta$ to denote its extension to $\FF_qG$. By Lemma III.1 in \cite{FGM}, we have that for the extension of $\theta$ on the group algebras, $$\theta(I_H)=I_{\theta(H)}=I_K$$ holds. That is $I_H$ is $G$-equivalent to $I_K$. 
 
Conversely, assume that $I_{H}=(\FF_qG)e_H$ and $I_{K}=(\FF_qG)e_K$ are $G$-equivalent.  This implies that there exists an automorphism $\theta$ of $G$
such that $\theta(I_{H})=I_{K}$. In this case, either $I_H=(\FF_qG)\widehat{G}=I_K$ or both $I_H$ and $I_K$ are different from $(\FF_qG)\widehat{G}$. By Lemma III.1 in \cite{FGM} we have that $$\theta(I_{H})= (\FF_qG) e_{\theta{(H)}}=I_K=(\FF_qG)e_K.$$ Then we have $e_K=e_Ke_{\theta{(H)}}=e_{\theta{(H)}}$ as $e_{\theta(H)}$ and $e_K$ are identity elements of  $(\FF_qG)e_K$. Since we have $e_K=e_{\theta{(H)}}$, it follows that $\theta(H)=K$. 
\end{proof}

The following result will be used frequently in the paper.

\begin{pro}[Proposition 1.1 of \cite{AT}]\label{prop11}
Let $G$ be a finite group and let $H$, $K$ be cocyclic subgroups of $G$. Then $H$ and $K$ are $G$-isomorphic if and 
only if they are isomorphic.
\end{pro}

Hence, the number of non-$G$-equivalent codes in $\mathcal{I}_{\FF_qG}$ is equal to the number of non-isomorphic 
subgroups in $\mathcal{S}_{cc}(G) \cup \{G\}$.
 Letting $\tau(G)$ denote the number of divisors of the exponent of $G$, it is easy to see that the number of non-isomorphic subgroups 
 in $\mathcal{S}_{cc}(G) \cup \{G\}$ is at least $\tau(G)$. Hence, Theorem 1.7 in \cite{AT} can be written as follows. 
 
\begin{thm}[Theorem 1.7 \cite{AT}] Let $G$ be a finite abelian group. The number of non-isomorphic subgroups in $\mathcal{S}_{cc}(G) \cup \{G\}$ is equal to 
the number of divisors of exponent of $G$ if and only if for each prime $p$ dividing the order of $G$, the Sylow $p$-subgroups of $G$ are homocyclic. 
\end{thm}
Recall that a {\bf{homocyclic group}} is a direct product of pairwise isomorphic cyclic groups. The theorem above implies the following result.

\begin{cor}\label{nonisoexist}

\textbf{(i)} Let $G$ be a finite abelian group. If one of Sylow $p$-subgroups of $G$ is not homocyclic, then there exist two cocyclic subgroups $H$ and $K$ such that $H$ is not isomorphic to 
$K$ and $|H|=|K|$. 

 \textbf{(ii)} Suppose that Sylow $p$-subgroups of $G$ are homocyclic and $H,K$  are cocyclic subgroups of $G$, then $H$ is isomorphic to $K$ if and only if $|H|=|K|$.
\end{cor}

\section{The case for $p$-groups} 
\begin{thm}\label{main1A} Let $p$ be an odd prime and let $q$ be a prime power where $(q, p)=1$. Assume that $G$ is an abelian $p$-group. Then the following are equivalent.

\textbf{(i)} Condition A holds for $\mathcal{I}_{\FF_q G}$;

\textbf{(ii)} Condition B holds for $\mathcal{I}_{\FF_q G}$;

\textbf{(iii)} $G$ is homocyclic.

\end{thm}

\begin{proof} $\textbf{(i)}\Rightarrow \textbf{(ii)}$ Assume Condition A holds for $\mathcal{I}_{\FF_q G}$. We need to prove that if weight distributions of $I_{H_1}, I_{H_2}\in \mathcal{I}_{\FF_q G}$ are identical, then $I_{H_1}$ and $I_{H_2}$ are $G$-equivalent. Assume that $I_{H_1}$ is not $G$-equivalent to $I_{H_2}$. As Condition A holds we have $w(I_{H_1})\neq w(I_{H_2})$. This means their weight distributions are different which gives a contradiction.

$\textbf{(ii)}\Rightarrow \textbf{(iii)}$  Suppose that Condition B holds for $\mathcal{I}_{\FF_q G}$. Suppose to the
 contrary that $G$ is not homocyclic. Then by Corollary \ref{nonisoexist} (i) there exists two non-isomorphic
cocyclic subgroups $H_1$ and $H_2$ of $G$ such that $|H_1|=|H_2|$. So  we have that $G/H_1 \cong G/H_2 \cong C_{p^r}$ for some
$r \geq 1$.
By Proposition \ref{prop11}, $H_1$ and $H_2$ are not $G$-isomorphic. Then by Lemma \ref{equivGiso}, $I_{H_1}$ and $I_{H_2}$ are
not $G$-equivalent. But since Condition B holds for $\mathcal{I}_{\FF_q G}$, the weight distribution of these codes
are different. Moreover since $e_{H_i} \widehat{H_i}=e_{H_i}$, we have that $I_{H_i} \subseteq (\FF_q G)\widehat{H_i}$ for $i=1, 2$. By considering the isomorphism $\varphi_{H_i}: (\FF_q G) \widehat{H_i} \to \FF_q(G/H_i)$
we get that $$\varphi_{H_i}(e_{H_i})=\widehat{H_i/H_i}-\widehat{H_i^*/H_i}=\widehat{1}-\widehat{C_p} \in \FF_qC_{p^r}$$ for $i=1, 2$.
Hence, $\varphi_{H_i}(I_{H_i})\cong (\FF_qC_{p^r}) (\widehat{1}-\widehat{C_p})$ for $i=1, 2$.  By Lemma \ref{reduction1}, the weight distributions
of $I_{H_1}$ and $I_{H_2}$ are identical. This is a contradiction. Therefore, $G$ is homocyclic.

$\textbf{(iii)}\Rightarrow \textbf{(i)}$
Assume $G$ is homocyclic and assume that for  codes $I_1, I_2 \in \mathcal{I}_{\FF_q G} $ we have $w(I_1)=w(I_2)$. Note that $w((\FF_qG)\widehat{G})=|G|$ and by Proposition 2.1 in \cite{DFM}, $w(\FF_qGe_H)=2|H|$ whenever $H$ is a cocyclic subgroup of $G$. Thus, since $|G|$ is odd, either both of $I_1$ and $I_2$ are $(\FF_qG)\widehat{G}$  or  there are  cocylic subgroups $H_1, H_2$ of $G$, so that 
$I_1=(\FF_q G) e_{H_1}$ and $I_2=(\FF_q G) e_{H_2}$. If the second case holds,  Proposition 2.1 in \cite{DFM} implies that $w(I_1)=2|H_1|=2|H_2|=w(I_2)$. As $H_1$ and $H_2$ are cocyclic subgroups of a homocyclic group and  $|H_1|=|H_2|$, we get that $H_1$ is isomorphic to $H_2$ by Corollary \ref{nonisoexist} (ii). Then Proposition \ref{prop11} implies that $H_1$ and $H_2$  are $G$-isomorphic. Hence, $I_1$ is $G$-equivalent to $I_2$ by Lemma \ref{equivGiso}.

\end{proof}

In the proof of the part $\textbf{(iii)}\Rightarrow \textbf{(i)}$ the condition on the prime $p$ is necessary  as the following example shows.

\begin{example}\label{prime2}
Let $G=\langle a \rangle\cong C_2$, and consider $\FF_3 G$. Then consider the ideals $I_1$ and $I_2$
generated by the idempotents corresponding to subgroups $G$ and $1$, which are equal to $2+2 a$ and $2+a$, respectively.
Then $I_1=\{ 0, 1+a, 2+2a\}$ and $I_2=\{0, 2+a, 1+2a\}$. Now, it is easy to see that the weights of $I_1$ and $I_2$
are equal but they are not $G$-equivalent by Lemma \ref{equivGiso}. Note
 that, $I_1$ and $I_2$ have identical
weight distribution and equal dimension. 

\end{example}

\begin{cor}{\label{qgenerator}}
Let $p$ be an odd prime, $G$ an abelian $p$-group with exponent $p^n$ and $q$ a prime power where $\langle q \rangle =U(\mathbb{Z}_{p^n})$. Then the following are equivalent.

\textbf{(i)} Condition A holds for  the set of minimal ideals of $\FF_q G$;

\textbf{(ii)} Condition B holds for the set of minimal ideals of $\FF_q G$;

\textbf{(iii)} $G$ is homocyclic.

\end{cor}

\begin{proof}
Follows from Theorem 4.1 of \cite{FM} and Theorem \ref{main1A}.
\end{proof}

\begin{thm}\label{dimension} Let $p$ be an odd prime,   $G$  an abelian $p$-group and $\FF_q$ a finite field of $q$ elements such that $(p,q)=1$. Condition C holds for ${\mathcal{I}}_{\FF_q G}$ if and only if $G$ is homocyclic.
\end{thm}
\begin{proof} Assume Condition C holds and suppose to the contrary that $G$ is not homocyclic. By Corollary \ref{nonisoexist} (i), 
 there are non-isomorphic cocyclic subgroups  $H_1$ and  $H_2$ of $G$  which have equal order. Let $I_1=\FF_q G e_{H_1}$ and $I_2=(\FF_q G) e_{H_2}$ be the corresponding ideals in $\mathcal{I}_{\FF_q G}$. Then by Lemma \ref{equivGiso}, $I_1$ is not $G$-equivalent to $I_2$.  By using Proposition 2.1 in \cite{FM} we have ${\rm{dim}}(I_1)={\rm{dim}}(I_2)$ which means Condition C does not hold for $\mathcal{I}_{\FF_q G}$.

Assume $G$ is homocyclic. Let $I_1$ and $I_2$ be two ideals of the same dimension in $\mathcal{I}_{\FF_q G}$. The  dimension of   $(\FF_qG)\widehat{G}$ is equal to  one and for a cocyclic subgroup $H$, the ideal $(\FF_qG)e_H$ is even dimensional as we have  ${\rm{dim}}((\FF_qG)e_H)=|G/H|-|G/{H}^*|$  by Proposition 2.1 in \cite{FM}.  Since $G$ has odd order, either both $I_1$ and $I_2$ are $(\FF_qG)\widehat{G}$ or both of them are  different from  $(\FF_qG)\widehat{G}$. If the second case holds, there are cocyclic subgroups $H_1$ and $H_2$ of $G$ so that $I_1=(\FF_qG)e_{H_1}$ and $I_2=(\FF_qG)e_{H_2}$. Then by using Proposition 2.1 in \cite{FM} we have ${\rm{dim}}(I_1)=|G/H_1|-|G/{H_1}^*|=|G/H_2|-|G/{H_2}^*|={\rm{dim}}(I_2)$. As $|H_i^*|=p|H_i|$ we can easily conclude that $|H_1|=|H_2|$. In this case as $H_1$, $H_2$ are cocyclic subgroups in a homocyclic group, Corollary \ref{nonisoexist} (ii) implies that they are isomorphic. hence they are $G$-isomorphic by Proposition \ref{prop11}. As a result, $I_1$ and $I_2$ are $G$-equivalent.
\end{proof}
\begin{cor}\label{qgenetarordim}
Let $p$ be an odd prime, $G$ an abelian $p$-group with exponent $p^n$ and $q$ a prime power where $\langle q \rangle =U(\mathbb{Z}_{p^n})$.
Then Condition C holds on the set of all minimal codes of $\FF_q G$ if and only if $G$ is homocyclic.
\end{cor}

\begin{proof}
Follows from Theorem 4.1 in \cite{FM} and Theorem \ref{dimension}.
\end{proof}

We should emphasize here that if we take the field as a splitting field for an abelian $p$-group which is homocyclic, Corollary \ref{qgenerator} and Corollary \ref{qgenetarordim} are not true anymore.
\begin{example} Let $G=C_3=<g>$ be the cyclic group of order  $3$.  Consider $\FF_4=\{0,1,\alpha, {\alpha}^2 \}$ where $\alpha$ is a primitive third root of unity.  $\FF_4$ is a splitting field for $G$. Then the character table of $G$ is

$$
\begin{tabular}{l|c|c|c|l}
$G$ & $1$         &$ g  $      & $g^2 $            &                            \\ \hline
$\chi_0$    & $1$ & $1$ & $1$     &    \\ \hline
$\chi_1$    & $1$ & $\alpha$ & ${\alpha}^2$ &    \\ \hline
$\chi_2$    & $1$ & ${\alpha}^2$ & $\alpha$  &   \\ \hline 

\end{tabular}
$$
 The primitive idempotents will be in the form $e_{\chi_i}=\frac{1}{3}\sum_{g\in G} \chi_i(g)g$ for $i=1,2,3.$
Then $e_{\chi_0}=1+g+g^2$ and $e_{\chi_1}=1+\alpha g+{\alpha}^2 g^2$. The corresponding minimal ideals are $$I_0=\{0,1+g+g^2, \alpha+\alpha g+\alpha g^2,{\alpha}^2+  {\alpha}^2 g+ {\alpha}^2 g^2 \}$$  
$$I_1=\{0,1+\alpha g+ {\alpha}^2 g^2,\alpha+{\alpha}^2 g+ g^2,{\alpha}^2+ g+ {\alpha}  g^2 \}$$

Now $w(I_0)=3=w(I_1)$, ${\rm{dim}}(I_1)={\rm{dim}}(I_2)$ and weight distributions of $I_0$ and $I_1$ are identical. 
As any automorphism of $G$ fixes any element of $I_0$, $I_1$ and $I_0$ are not $G$-equivalent. So none of the Conditions $A, B, C$ holds for this example.

\end{example}

 We have a natural question:

\begin{que}\label{que} Let $G$ be  an abelian $p$-group and $\FF$  a finite splitting field for $G$ of characteristic coprime to the  order of $G$. What are the conditions on $G$ so that minimal ideals in $\FF G$ satisfies Condition A,  Condition B and Condition C hold ?

\end{que}

Let $\mathcal{I}_{min}(\FF G)$ be the set of minimal ideals in $\mathbb{F}G$ generated by the primitive  idempotents corresponding to the non-trivial irreducible characters. That is we disclude the minimal ideal generated by $e_{\chi_0}$ where $\chi_0$ is the trivial character. We give an answer for the Question \ref{que} as follows.

\begin{thm} Let $G$ be an abelian $p$-group and let $\FF$ be a finite spliting field for $G$ of characteristic coprime to the order of $G$. Then the following are equivalent.

\textbf{(i)} Condition A holds on $\mathcal{I}_{min}(\FF G)$;

\textbf{(ii)} Condition B holds on $\mathcal{I}_{min}(\FF G)$;

 \textbf{(iii)} $G$ is an elementary abelian $p$-group.

\end{thm}

\begin{proof} Assume Condition A holds. Let $I_1, I_2$ be two ideals in $\mathcal{I}_{min}(\FF G)$ having identical weight distributions. Suppose to the contrary that $I_1$ is not $G$-equivalent to $I_2$. As Condition A holds, we have $w(I_1)\neq w(I_2)$. This is a contradiction as they have same weight distribution. So $I_1$ is $G$-equivalent to $I_2$, hence Condition B holds.

Assume Condition B holds. Assume $G$ is not elementary abelian. Then $G$ has a direct factor of order $p^n$ where $n>1$. Then there are elements $g, h \in G $ such that $h\notin \langle g\rangle$ and $g$ has order $p^n$ and $h$ has order $p$. Recall that for an irreducible character $\chi$ of $G$, the corresponding idempotent is  $e_{\chi}=\frac{1}{|G|}\sum_{g\in G} \chi(g)g$. Let $\chi_g$ be the irreducible character with $\chi_g(g)=\alpha$ and $\chi_g(h)=1$ and let $\chi_h$ be the irreducible character with $\chi_h(g)=1$ and $\chi_h(h)=\beta$ where $\alpha$ is primitive $p^n$-th root of unity and $\beta$ is  primitive   $p$-th root of unity. Then it is easy to see that the idempotents $e_{\chi_g}$ and $e_{\chi_h}$ generate non-$G$-equivalent minimal codes. However these minimal codes have identical weight distribution by the classfication of irreducible representations of abelian groups  over their splitting field. This means Condition $B$ does not hold which is a contradiction.

Assume $G$ is an elementary abelian $p$-group. By considering its non-trivial characters and corresponding primitive idempotents, we can conclude that each minimal ideal in
 $\mathcal{I}_{min}(\FF G)$  have weight equal to $|G|$ and since every non identity element can be sent to another non-identity element of $G$ by an automorphism of $G$, all these minimal ideals are $G$-equivalent to each other. Hence  Condition A holds.
\end{proof}

\begin{thm} Let $G$ be an abelian $p$-group and let $\FF$ be a finite spliting field for $G$ of characteristic coprime to the order of $G$. Condition C holds on $\mathcal{I}_{min}(\FF G)$ if and only if $G$ is elementary abelian.
\end{thm}
\begin{proof} Assume $G$ is an elemantary abelian $p$-group. By classification of irreducible representations of abelian groups over their splitting field, we can conclude that all minimal ideals in $\mathcal{I}_{min}(\FF G)$ are one dimensional. Moreover, each minimal ideal in $\mathcal{I}_{min}(\FF G)$ is $G$-equivalent to another minimal ideal in $\mathcal{I}_{min}(\FF G)$. This is because any non-identity element can be sent to another non-identity element of $G$ by an automorphism of $G$.  That is Condition C holds. 

Conversely assume that Condition C holds and assume $G$ is not elementary abelian $p$-group. Then $G$ has a direct factor isomorphic to $C_{p^n}$ where $n>1$.  As in the proof of the previous result, we can conclude that all  minimal ideals are one dimensional but there are non-$G$-equivalent ones.
\end{proof}

\section{The case for composite  groups}

The following lemma can be seen as a generalization of Proposition 2.3 (ii) of \cite{DFM}.
Note that in some results below we use the sign $\prod$ for both direct product of groups and product of elements in the form $\widehat{H}$ for some subgroup $H$ of the given group and numbers.

\begin{lem}\label{reduction2}
Let $G=G_{p_1}\times \ldots \times G_{p_k}$ be an abelian group where each $G_{p_i}$ is a Sylow $p_i$-subgroup of $G$. Let
$H$ be a cocyclic subgroup of $G$ and write $H$ as
$$H=(\prod_{i\in S}G_{p_i}) \times (\prod_{i\in \{1, \ldots, k\} \backslash S} H_{p_i})$$
where $S\subseteq \{1, \ldots, k\}$ and for each $i\in \{1, \ldots, k\} \backslash S$, $H_{p_i}$ is a cocyclic subgroup of $G_{p_i}$.
Consider $I_H= (\FF_q G) e_H$ where $(q, p_i)=1$ for any $i\in \{1, \ldots, k\}$. Then
$$w(I_H)= 2^{k-|S|} |H|.$$
\end{lem}

\begin{proof}
Let $T=(\prod_{i\in S}G_{p_i})\leq H$, then we have that $\widehat{T} =(\prod\limits_{i\in S}\widehat{G_{p_i}})$ clearly. So
$$e_H \widehat{T}=\bigg{(}\prod\limits_{i\in S}\widehat{G_{p_i}}\bigg{)} \bigg{(}\prod\limits_{i\in \{1, \ldots, k\} \backslash S}
 (\widehat{H_{p_i}}-\widehat{H_{p_i}^*})\bigg{)}\bigg{(}\prod\limits_{i\in S}\widehat{G_{p_i}}\bigg{)}=e_H$$
and it follows that $I_H \subseteq (\FF_q G) \widehat{T}$. Note that
$$\varphi_T(e_H)=\bigg{(}\prod\limits_{i\in \{1, \ldots, k\} \backslash S}
 (\widehat{H_{p_i}}-\widehat{H_{p_i}^*})\bigg{)}=e_{H/T}$$
 implies that $\varphi_T(I_H)=I_{H/T}\subseteq \FF_q(G/T)$. From Proposition 2.3 (ii) of \cite{DFM},
 it follows that $$w(I_{H/T})=2^{k-|S|}|H|/|T|$$
 and then Lemma \ref{reduction1} implies  that $w(I_H)=2^{k-|S|}|H|$.
\end{proof}

\begin{thm}\label{main2A}
Let $n$ be an odd integer and let $G$ be an abelian group order n. Let $q$ be a prime power with $(q, n)=1$. Then Condition A holds for
$\mathcal{I}_{\FF_q G}$ if and only if every Sylow $p$-subgroup of $G$ is homocyclic.
\end{thm}

\begin{proof}
Let $G=G_{p_1}\times \ldots \times G_{p_k}$ where each $G_{p_i}$ is a Sylow $p_i$-subgroup of $G$. Assume that $G_{p_i}$
is homocyclic for each $i=1, \ldots, k$. Suppose that the weights of $I_{H}$ and $I_{K}$ are  equal. If one of $H$ or $K$ is equal to
$G$, then the other one should also be equal to $G$. Indeed, since $G$ has odd order, the equality of weights and Lemma \ref{reduction2}
imply that if $H=G$ then $K=G$. So, we can
assume without loss of generality that both of $H$ and $K$ are different from $G$. In this case again Lemma \ref{reduction2}
implies that
$$2^{k-|S|} |H| =w(I_H)=w(I_K)=2^{k-|S'|}|K|$$
and since $G$ has odd order, it follows that $|H|=|K|$. Because $H$ and $K$ are cocyclic subgroups of $G$ and each Sylow
subgroup of $G$ is homocyclic, it follows that $H\cong K$ by Corollary \ref{nonisoexist} (ii). Hence, by Proposition \ref{prop11}, we have that $H$ is $G$-isomorphic
to $K$. Therefore, by Lemma \ref{equivGiso}, we get that $I_H$ is $G$-equivalent to $I_K$. Thus, Condition A holds for
$\mathcal{I}_{\FF_q G}$.

Suppose that Condition A holds for $\mathcal{I}_{\FF_q G}$ and suppose to the contrary that at least one of Sylow $p$-subgroups
is not homocyclic, say $G_{p_j}$. Then by Corollary \ref{nonisoexist}(i), there exist two non-isomorphic cocyclic subgroups $H_{p_j}$ and $K_{p_j}$ of $G_{p_j}$ such that
$|H_{p_j}|=|K_{p_j}|$. Consider, $H=H_{p_j} \times (\prod_{i\neq j}G_{p_i})$ and $K=K_{p_j} \times (\prod_{i\neq j}G_{p_i})$. Then
since $|H|=|K|$  we have that $w(I_H)=w(I_K)$ by Lemma \ref{reduction2}. However, since $H$ is not isomorphic to $K$,
by Proposition \ref{prop11}, we have that $H$ is not $G$-isomorphic to $K$. Then 
Lemma \ref{equivGiso} implies that
$I_H$ is not $G$-equivalent to $I_K$, that is Condition A does not hold for $\mathcal{I}_{\FF_q G}$, which is a contradiction. So each
Sylow $p$-subgroup of $G$ should be homocyclic.

\end{proof}

Likewise as in the $p$-group case, we did not use the fact that $G$ has odd order in the forward direction of the proof of the theorem above.
But in the reverse direction, we need the odd order condition as the following example shows.

\begin{example}
Let $G=\langle a \rangle \cong C_6$ and consider $\FF_5 G$. Then consider the ideals $I_1$ and $I_2$ generated by the
idempotents $\widehat{G}$ and $e_{\langle a^2 \rangle}= \widehat{\langle a^2 \rangle} \ (1- \widehat{\langle a^3 \rangle})$,
respectively. Then by Lemma 3.1, we have that $w(I_2)=6$. Notice that $w(I_1)=6$, also $I_1$ and $I_2$ are not
$G$-equivalent since $G$ and $\langle a^2 \rangle$ are not $G$-isomorphic. So even though every Sylow $p$-subgroup
of $G$ is homocyclic, Condition A is not satisfied by $\mathcal{I}_{\FF_5 G}$.
\end{example}

\begin{thm}\label{main2B}
Let $n$ be an odd integer and let $G$ be an abelian group of order n. Let $q$ be a prime power with $(q, n)=1$. Then Condition B holds for
$\mathcal{I}_{\FF_q G}$ if and only if every Sylow $p$-subgroup of $G$ is homocyclic.
\end{thm}

\begin{proof}
Let $G=G_{p_1}\times \ldots \times G_{p_k}$ be an abelian group where each $G_{p_i}$ is a Sylow $p_i$-subgroup of $G$. Assume that $G_{p_i}$
is homocyclic for each $i=1, \ldots, k$. Suppose that the weight distribution of $I_H$ and $I_K$ are identical, then the weight of
$I_H$ and $I_K$ are equal. Then, it follows from Theorem \ref{main2A} that $I_H$ and $I_K$ are $G$-equivalent. So
Condition B holds for $\mathcal{I}_{\FF_q G}$.

Assume that Condition B holds for $\mathcal{I}_{\FF_q G}$. Suppose to the contrary that $G_{p_j}$ is not homocyclic
for some $1 \leq j \leq k$.
Then by Corollary \ref{nonisoexist}(i), there exist two non-isomorphic cocyclic subgroups $H_{p_j}$ and $K_{p_j}$ of $G_{p_j}$ such that
$|H_{p_j}|=|K_{p_j}|$. Consider, $H=H_{p_j} \times (\prod_{i\neq j}G_{p_i})$ and $K=K_{p_j} \times (\prod_{i\neq j}G_{p_i})$. Then
$H$ and $K$ are cocyclic subgroups of $G$ such that $H\not\cong K$ and $|H|=|K|$. In this case $H$ is not $G$-isomorphic
to $K$ by Proposition \ref{prop11}. Then Lemma \ref{equivGiso}  implies that
$I_H$ is not $G$-equivalent to $I_K$. Since Condition B holds, weight distributions of $I_H$ and $I_K$ are different.

Now, note that since $e_H \widehat{H}=e_H$, we have that $I_H=\FF_q Ge_H \subseteq (\FF_q G) \widehat{H}$ and
similarly $I_K=\FF_q Ge_K \subseteq (\FF_q G) \widehat{K}$. Moreover, since
$$e_H=(\widehat{H_{p_j}}-\widehat{H_{p_j}^*})   (\prod\limits_{i\neq j}\widehat{G_{p_i}}) \text{ and } e_K=(\widehat{K_{p_j}}-\widehat{K_{p_j}^*})   (\prod\limits_{i\neq j}\widehat{G_{p_i}}) $$ and $G/H\cong G/K\cong C_{{p_j}^r}$ for some integer $r\geq 1$.
We have that
$$\varphi_H(e_H)=1-\widehat{C_{{p_j}}} \text{ and } \varphi_K(e_K)=1-\widehat{C_{{p_j}}}. $$
This implies that the weight distributions of $\varphi_H(I_H)$ and $\varphi_K(I_K)$ are identical.
On the other hand since the  weight distributions of $I_H$ and $I_K$ are different and $|H|=|K|$, Lemma \ref{reduction1}
implies that the weight distributions of $\varphi_H(I_H)$ and $\varphi_K(I_K)$ are different. This is a contradiction. Hence, we
conclude that each Sylow $p$-subgroup of $G$ is homocyclic.
\end{proof}

\bigskip

The following lemma is a kind of generalization  of Proposition 2.3 (i) in \cite{DFM}.

\begin{lem}\label{dimeH} Let $G=G_{p_1}\times \ldots \times G_{p_k}$ be an abelian group where each $G_{p_i}$ is a Sylow $p_i$-subgroup of $G$. Let
$H$ be a cocyclic subgroup of $G$ and write $H$ as
$$H=(\prod_{i\in S}G_{p_i}) \times (\prod_{i\in \{1, \ldots, k\} \backslash S} H_{p_i})$$
where $S\subseteq \{1, \ldots, k\}$ and for each $i\in \{1, \ldots, k\} \backslash S$, $H_{p_i}$ is a cocyclic subgroup of $G_{p_i}$.
Consider $I_H= (\FF_q G) e_H$ where $(q, p_i)=1$ for any $i\in \{1, \ldots, k\}$. Then
 
$${\rm{dim}}(I_H)=\frac{|G|}{|H|} \prod\limits_{i\in \{1, \ldots, k\} \backslash S}(1-\frac{|H_{p_{i}}|}{|H_{p_{i}}^*|})=\frac{|G|}{|H|} \prod\limits_{i\in \{1, \ldots, k\} \backslash S}(1-\frac{1}{p_{i}})$$
In particular, if $H,K$ are cocyclic subgroups of $G$ such that $|H|=|K|$, then ${\rm{dim}}(I_H)={\rm{dim}}(I_K)$.
\end{lem}

\begin{proof} As  in the proof of Lemma \ref{reduction2}, let 
 $T=(\prod_{i\in S}G_{p_i})\leq H$. It is clear that we have  
$e_H \widehat{T}=e_H$ and $I_H \subseteq (\FF_q G) \widehat{T}$. As 
$H/T\cong \prod\limits_{{i\in \{1, \ldots, k\} \backslash S}}H_{p_i}$ and 
$G/T\cong \prod\limits_{i\in \{1, \ldots, k\} \backslash S}G_{p_i}$, it follows that $H/T$ is a cocyclic subgroup of $G/T$. So  we have 
$e_{H/T}=\prod\limits_{{i\in \{1, \ldots, k\} \backslash S}} e_{H_{p_i}} \in \FF_q(G/T)$. Then by Proposition 2.3 in \cite{DFM}, we have 
$${\rm{dim}}(\FF_q(G/T))e_{H/T}= \frac{|G/T|}{|H/T|}\prod\limits_{i\in \{1, \ldots, k\} \backslash S}(1-\frac{1}{p_i})=\frac{|G|}{|H|}\prod\limits_{i\in \{1, \ldots, k\} \backslash S}(1-\frac{1}{p_i}).$$
Since $\FF_q(G/T)e_{H/T}=\varphi_T((\FF_qG)e_H)=\varphi_T(I_H)$ and $\varphi_T$ preserves the dimension, we have $${\rm{dim}}(I_H)=\frac{|G|}{|H|} \prod\limits_{i\in \{1, \ldots, k\} \backslash S}(1-\frac{1}{p_{i}}).$$

If $H,K$ are cocyclic subgroups of $G$ such that $|H|=|K|$, by the dimension formula above, we get ${\rm{dim}}(I_H)={\rm{dim}}(I_K)$.

\end{proof}

The converse of the last statement of the previous theorem is not true in general.

\begin{example}\label{counterex} Let $G=<a>\times <b>\times <c>\cong C_9\times C_9 \times C_{49}$. Consider the cocyclic subgroups  $H=<a>\times <c>\cong C_9\times C_{49}$ and $K=<a>\times <b>\times <c^2>\cong C_9\times C_9\times C_7$ of $G$. By Lemma \ref{dimeH} we have ${\rm{dim}}(I_H)=\frac{|G|}{|H|}(1-\frac{1}{3})=6$
 and ${\rm{dim}}(I_K)=\frac{|G|}{|K|}(1-\frac{1}{7})=6$. On the other hand we have $|K|\neq |H|$.
\end{example}

However there is an infinite family of abelian groups which satisfy the desired property if we add some extra conditions. Note that the set of tuples $(p_1, p_2)$ in the next 
theorem are infinite since for example there 
are infinitely many prime tuples of the form $(3, 3k+2)$.

\begin{pro} \label{equaldim} Let $p_1, p_2$ be odd primes such that $p_1<p_2$ and 
$p_2\not\equiv 1({\rm{mod}}~p_1)$. Let $G=G_{p_1}\times G_{p_2}$ be an  abelian group where each $G_{p_i}$ is Sylow $p_i$-subgroup of $G$ and $H,K$ cocyclic subgroups of $G$. Then  if ${\rm{dim}}(I_H)={\rm{dim}}(I_K)$, then  $|H|=|K|$. 
\end{pro}

\begin{proof} 

 There are cases we need to consider.

$\bf{Case~1}:$ Let $H=H_{p_1}\times H_{p_2}$ and $K=K_{p_1}\times K_{p_2}$ where  $H_{p_i}, K_{p_i} < G_{p_i}$ for $i \in \{1,2\}$ are cocyclic subgroups. 
Let $h_i, k_i\geq 1$ be the integers such that $|G_{p_i}: H_{p_i}|={p_i}^{h_i}$ and $|G_{p_i}: K_{p_i}|={p_i}^{k_i}$. Then by Lemma \ref{dimeH}, we have 
$${\rm{dim}}(I_H)={p_1}^{h_1}{p_2}^{h_2}(1-\frac{1}{p_1})(1-\frac{1}{p_2})={p_1}^{k_1}{p_2}^{k_2}(1-\frac{1}{p_1})(1-\frac{1}{p_2})={\rm{dim}}(I_K).$$  
It follows that  $h_i=k_i$ for $i\in \{1,2\}$. So $|H|=|K|$.

$\bf{Case ~2}:$ $H=H_{p_1}\times G_{p_2}$ and $K=K_{p_1}\times G_{p_2}$ where for  $H_{p_1}, K_{p_1} <G_{p_1}$ are cocyclic subgroups. If we have  $|G_{p_1}: H_{p_1}|={p_1}^{h_1}$ and $|G_{p_1}: K_{p_1}|={p_1}^{k_1}$, then we get 
$${\rm{dim}}(I_H)={p_1}^{h_1}(1-\frac{1}{p_1})={p_1}^{k_1}(1-\frac{1}{p_1})={\rm{dim}}(I_K).$$
So we get $h_1=k_1$ and $|H|=|K|$.

Similarly if $H=G_{p_1}\times H_{p_2}$ and $K=G_{p_1}\times K_{p_2}$ where for  $H_{p_2}, K_{p_2} <G_{p_2}$ are cocyclic subgroups such that ${\rm{dim}}(I_H)={\rm{dim}}(I_K)$, then $|H|=|K|$.

$\bf{Case ~3}:$ If $H=H_{p_1}\times H_{p_2}$ and $K=G_{p_1}\times K_{p_2}$ where $H_{p_1}<G_{p_1}$ and   $H_{p_2}, K_{p_2} <G_{p_2}$ are cocyclic subgroups, by Lemma \ref{dimeH}, ${\rm{dim}}(I_H)\neq{\rm{dim}}(I_K)$. It follows similarly if $H=H_{p_1}\times H_{p_2}$ and $K=K_{p_1}\times G_{p_2}$.

$\bf{Case ~4}:$ Let $H=H_{p_1}\times G_{p_2}$ and $K=G_{p_1}\times K_{p_2}$ where for  $H_{p_1}  <G_{p_1}$, $H_{p_2} <G_{p_2}$ are cocyclic subgroups and $|G_{p_1}: H_{p_1}|={p_1}^{h_1}$ and $|G_{p_2}: K_{p_2}|={p_2}^{k_2}$. Then  we have ${\rm{dim}}(I_H)={p_1}^{h_1}(1-\frac{1}{p_1})$, ${\rm{dim}}(I_K)={p_2}^{k_2}(1-\frac{1}{p_2})$ by Lemma \ref{dimeH}. 
Assume ${\rm{dim}}(I_H)={\rm{dim}}(I_K)$. Then we have ${p_1}^{h_1-1}(p_1-1)={p_2}^{k_2-1}(p_2-1)$.
If $h_1> 1$, then $p_1$ should divide $p_2^{k_2}$ which is a contradiction.
If $h_1=1$, then $p_1-1={p_2}^{k_2-1}(p_2-1)$ which is not possible as $p_1<p_2$. So ${\rm{dim}}(I_H)\neq{\rm{dim}}(I_K)$. 

\end{proof}

\begin{pro} Let $p_1, p_2$ be odd primes such that $p_1<p_2$ and 
$p_2\not\equiv 1({{\rm{mod}} ~p_1})$. Let $G=G_{p_1}\times G_{p_2}$ be an abelian group where each $G_{p_i}$ is Sylow $p_i$-subgroups of $G$. If $G_{p_1}$ and $G_{p_2}$ are homocyclic, then Condition C holds for $\mathcal{I}_{\FF_qG}$.
\end{pro} 

\begin{proof} Assume $G_{p_1}$ and $G_{p_2}$ are homocyclic. Let $I_H$ and $I_K$ be   ideals in $\mathcal{I}_{\FF_qG}$. As ${\rm{dim}}(\FF_q\widehat{G})=1$ and order of $G$ is odd, by dimension formula in Lemma \ref{dimeH}, either both $I_H$ and $I_K$ are equal to $\FF_q\widehat{G}$ or both of $I_H$ and $I_K$ are different from $\FF_q\widehat{G}$. Assume that  ${\rm{dim}}(I_H)={\rm{dim}}(I_K)$.  By Propostion \ref{equaldim}, we have $|H|=|K|$.  As  the Sylow $p$-subgroups of $G$ are homocyclic, by Corollary \ref{nonisoexist}(ii) we have that $H\cong K$. Then by Proposition \ref{prop11}, $H$ is $G$-isomorphic to $K$. So by Lemma \ref{equivGiso}, $I_H$ is $G$-equivalent to $I_K$. Hence 
Condition C holds for $\mathcal{I}_{\FF_q G}$. 
\end{proof}

We have also prove the following theorem.

\begin{thm}\label{dimC}
Let $n$ be an odd integer and let $G$ be an abelian group order $n$. Let $q$ be a prime power with $(q, n)=1$. If  Condition C holds for
$\mathcal{I}_{\FF_q G}$, then  every Sylow $p$-subgroup of $G$ is homocyclic.
\end{thm}

\begin{proof}

Assume that Condition C holds for $\mathcal{I}_{\FF_q G}$. Suppose to the contrary that one of the Sylow $p$-subgroup  is not homocyclic.
Then by Corollary \ref{nonisoexist} (i), there exist two non-isomorphic cocyclic subgroups 
$H$ and $K$  of $G$ such that $H\not\cong K$ and $|H|=|K|$. In this case $H$ is not $G$-isomorphic
to $K$ by Proposition \ref{prop11}. Then Lemma \ref{equivGiso}  implies that
$I_H$ is not $G$-equivalent to $I_K$. Since Condition C holds, their dimensions are not equal. On the other hand by Lemma \ref{dimeH}, we have 
${\rm{dim}}(I_H)={\rm{dim}}(I_K)$  which gives a contradiction.
\end{proof}

We end the paper with the following question.
\begin{que} Let $G$ be an abelian group of odd order and $\FF_q$  a field  such that $\FF_qG$ is semisimple. Assume $\FF_q$ is not a splitting field for $G$.
Is there any other set  $\mathcal{I}$  of  ideals in $\FF_qG$ such  that Conditions A,  B holds on $\mathcal{I}$ if and only if for each prime divisor $p$ of $|G|$ the Sylow $p$-subgroup of $G$ is homocyclic?
\end{que}

\subsection*{Acknowledgements} The authors were partially supported by Mimar Sinan Fine Arts University 
Scientific Research Unit with project number 2019-27.


{\small

}

\bibliographystyle{amsplain}

\end{document}